\documentclass[a4paper]{amsart}

\usepackage[foot]{amsaddr}
\usepackage[utf8]{inputenc}
\usepackage[nocompress]{cite}
\usepackage{graphicx}
\usepackage{subcaption}
\usepackage{epsf, tikz, tkz-graph}
\usepackage{amssymb,amsmath,eufrak,amsthm,amscd,mathtools}
\usepackage{graphics}
\usepackage{textcomp}
\usepackage{comment}
\usepackage{enumitem}
\usepackage{verbatim}
\usepackage{csquotes, nicefrac}
\usepackage{capt-of,caption}
\usepackage{graphicx,newfloat,float}

\newtheorem{theorem}{Theorem}[section]

\newtheorem*{statement*}{Statement}
\newtheorem*{theorem*}{Theorem}
\newtheorem*{lemma*}{Lemma}
\newtheorem*{fact*}{Fact}

\theoremstyle{definition}
\newtheorem{definition}[theorem]{Definition}
\newtheorem*{definition*}{Definition}

\newtheorem*{example*}{Example}

\newtheorem*{exercise*}{Exercise}

\newtheorem*{proposition*}{Proposition}

\newtheorem*{corollary*}{Corollary}

\newtheorem*{claim*}{Claim}

\newtheorem*{test*}{Test}

\theoremstyle{remark}

\newtheorem*{remark*}{Remark}

\makeatletter
\newcommand{\l@abcd}[2]{\hbox to\textwidth{#1\dotfill #2}}

\newcommand{\e}{\varepsilon}

\newcommand{\f}{\varphi}

\title[Graphical Designs and Extremal Combinatorics]{Graphical Designs and Extremal Combinatorics}
\author[K. Golubev]{Konstantin Golubev}
\email{golubevk@ethz.ch}
\address{D-MATH, ETH Zurich, Switzerland}
\address{Moscow Center for Fundamental and Applied Mathematics, Russia}
\date{\today}

\begin{document}

\maketitle

\begin{abstract}A graphical design is a proper subset of vertices of a graph on which many eigenfunctions of the Laplacian operator have mean value zero. In this paper, we show that extremal independent sets make extremal graphical designs, that is, a design on which the maximum possible number of eigenfunctions have mean value zero. We then provide examples of such graphs and sets, which arise naturally in extremal combinatorics. 
We also show that sets which realize the isoperimetric constant of a graph make extremal graphical designs, and provide examples for them as well. We investigate the behavior of graphical designs under the operation of weak graph product. In addition, we present a family of extremal graphical designs for the hypercube graph.
\end{abstract}
\section{Introduction}
Let $G$ be a $d$-regular graph on the vertex set $V$ of size $n$ with no loops or multiple edges. Let $A$ denote the normalized adjacency operator on the space $\mathbb{R}^{V}$ of real-valued functions on the vertices of $G$. That is, for $f\in \mathbb{R}^{V}$ and a vertex $v\in V$
\[
 (Af)(v) = \frac{1}{d} \sum_{u\sim v} f(u),
\]
where $u\sim v$ means that $u$ and $v$ are connected by an edge in $G$. The adjacency operator is symmetric bounded operator of norm $1$, hence its eigenvalues are real and satisfy
\[
 -1\leq \lambda_n\leq \dots\leq \lambda_2 \leq \lambda_1 = 1.
\]
The graph is connected iff the eigenvalue $1$ has multiplicity $1$ (see, e.g.~\cite{alon1986eigenvalues}). From now on, we assume that $G$ is connected, and in this case the only eigenfunction with eigenvalue $1$ is the constant function. We can endow $\mathbb{R}^{V}$ with the standard inner product, defined as, for $f,g\in \mathbb{R}^{V}$
\[
\langle f,g \rangle = \sum_{v\in V} f(v)g(v).
\]
The problem of \emph{graphical designs} was posed in~\cite{steinerberger2018generalized} as follows.
 Let $G$ be a finite, simple, connected graph. Suppose there exists a subset $W\subset V$ with weights $a_w\in\mathbb{R}$ for $w\in W$ such that $\forall\, 1\leq k \leq K$
\begin{equation}\label{eq:graph-des}
\frac{1}{|V|}\sum_{v\in V}\varphi_k(v) = \sum_{w\in W} a_w \varphi_k(w),  
 \end{equation}
 where $\varphi_k$ is an eigenfunction of $A$ with eigenvalue $\lambda_k$. 
How big can $K$ be (depending on $|W|$)? How does this depend on $G$? How would one find sets $W$ having $K$ large?

In~\cite{steinerberger2018generalized}, the Laplace operator $\Delta$ of a graph was considered, defined as
\[
 (\Delta f)(v) = \sum_{u\sim v}\left(\frac{f(u)}{\deg(u)} - \frac{f(v)}{\deg(v)}\right).
\]
For a regular graph, the Laplacian $\Delta$ and the normalized adjacency operator $A$ satisfy
\[
 \Delta = A - I,
\]
where $I$ stands for the identity operator. In particular, these operators share eigenfunctions.

\begin{figure}[h!]\centering
  \begin{tikzpicture}[scale=0.7]

\foreach \a in {1,2,...,9}{

\draw (\a*360/9: 3cm) --  (\a*360/9 + 360/9: 3cm);

};

\draw  (0.8, -0.4) --  (8*360/9: 3cm);

\draw  (0.8, -0.4) -- (-0.8, -0.4) -- (0,1) -- (0.8, -0.4);

\draw  (-0.8, -0.4) --  (5*360/9: 3cm);

\draw  (0,1) --  (2*360/9: 3cm);

\draw  (1*360/9: 3cm) --  (3*360/9: 3cm);

\draw  (4*360/9: 3cm) --  (6*360/9: 3cm);

\draw  (7*360/9: 3cm) --  (9*360/9: 3cm);

\foreach \a in {1,2,...,9}{

\node[circle,draw=black, fill=white, inner sep=0pt,minimum size=6pt] () at (\a*360/9: 3cm) {};
};
\node[circle,draw=black, fill=white, inner sep=0pt,minimum size=6pt] () at (0.8, -0.4) {};
\node[circle,draw=black, fill=white, inner sep=0pt,minimum size=6pt] () at (0, 1) {};

\node[circle,draw=black, fill=black, inner sep=0pt,minimum size=6pt] () at (1*360/9: 3cm) {};
\node[circle,draw=black, fill=black, inner sep=0pt,minimum size=6pt] () at (5*360/9: 3cm) {};
\node[circle,draw=black, fill=black, inner sep=0pt,minimum size=6pt] () at (9*360/9: 3cm) {};
\node[circle,draw=black, fill=black, inner sep=0pt,minimum size=6pt] () at (-0.8, -0.4) {};
\end{tikzpicture}
\caption{The Truncated Tetrahedral Graph on 12 vertices: there exists an orthonormal basis of eigenfunctions, such that $10$ non-constant eigenfunctions of it have mean value zero on the subset of $4$ black vertices. (Courtesy of Steinerberger,~\cite{steinerberger2018generalized}.)}
\end{figure}
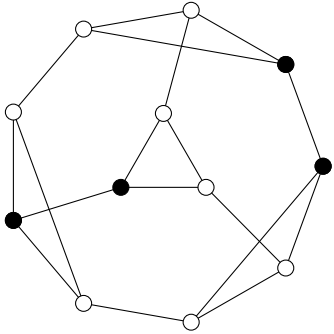
As it is pointed out in the same paper, there is some ambiguity when eigenvalues have large multiplicity. The problem of graphical designs was inspired by an analogous question on the sphere, and while on the sphere the enumeration of eigenfunctions sometimes is taken into account, it is not yet justified why it should matter so for graphs. 
In this paper, we study a certain subclass of graphical designs, which we baptize \emph{extremal}, and connect it to a well-developed branch of extremal combinatorics. We also narrow our consideration to the case of graphical designs $W$ of constant weights $a_w = {1}/{|W|}$.

\begin{definition} A graphical design on a graph $G$ is a proper subset $W\subset V$. Its order is the minimal such $k$ that the characteristic function $1_W$ of $W$ is a linear combination of the constant function and $k$ other eigenfunctions of $A$. A graphical design of order $1$ is called extremal.
\end{definition}
The justification for such a definition is that for a graphical design of order $k$ there exists a basis of eigenfunctions with respect to which it satisfies Equation~(\ref{eq:graph-des}) for all but $k$ eigenfunctions in the basis. Namely, if $1_W$ is a linear combination of the constant function and eigenfunctions $\varphi_{n_1},\dots,\varphi_{n_k}$, where $1<n_1<n_2\dots <n_k$, then one can complete them to an orthogonal basis of eigenfunctions $\varphi_1,\dots,\varphi_n$. Then for $1 < k \leq n$ such that $k\not\in \{n_1,\dots,n_k\}$, we have
\[
 \frac{1}{|V|}\sum_{v\in V}\varphi_k(v) = 0 = \frac{1}{|W|} \langle \varphi_k, 1_W \rangle = \frac{1}{|W|}\sum_{w\in W} \varphi_k(w).
\]
Equation~(\ref{eq:graph-des}) is also satisfied for $k=1$ as $\varphi_1$ is a constant function and it has mean value $1$ for every non-empty subset of vertices. 
% \[
%   \frac{1}{|V|}\sum_{v\in V}\varphi_1(v) = 1 =  \frac{1}{|W|}\sum_{w\in W} \varphi_1 (w).
% \]
%If Equation~(\ref{eq:graph-des}) holds for a function $\f_k$ on a set $W$, we say that $W$ integrates $\f_k$ exactly.

In this paper, we show that extremal graphical designs can be found in maximal independent sets and in subsets realizing the isoperimetric constant, and provide infinite families of examples arising naturally in extremal combinatorics.   

\subsubsection*{Relation to existing work.}
The term \textit{graphical design} was suggested in~\cite{steinerberger2018generalized} in analogy to spherical designs. Let $S^d$ be the unit sphere in $\mathbb{R}^{d+1}$ with the Lebesgue measure $\mu_d$ normalized by $\mu_d(S^d)=1$. A set of points $x_1,\dots,x_N\in S^d$ is called a spherical $t$-design if
\[
 \int_{S^d} P(x) d\mu_d = \frac{1}{N}\sum_{i=1}^N P(x_i)
\]
for all algebraic polynomials in $(d + 1)$ variables, of total degree at most t. This notion was introduced by Delsarte, Goethals, and Seidel in~\cite{delsarte1991spherical}, where they proved a lower bound on the cardinality of a spherical $t$-design in terms of $d$ and $t$. Spherical designs were however studied earlier, as for example, in~\cite{mclaren1963optimal}, McLaren constructed a set of $72$ points on $S^2$ integrating $225$ polynomials exactly. Among recent results is the work of Bondarenko, Radchenko and Viazovska,~\cite{bondarenko2013optimal}, where they prove asymptotic optimality of the lower bound on the size of a design. This result was later extended to manifolds in~\cite{gariboldi2018optimal}. Both spherical and graphical designs deal with the question of sampling: how to efficiently sample a domain in such way that many functions integrate exactly on the sample. Spectral approach to sampling in graphs is taken, for example, in~\cite{pesenson2008sampling} by Pesenson. In this work, Pesenson introduces a notion of a Paley-Wiener space for a graph and proves analogues of the classical results of Paley-Wiener theory for certain families of graphs. In particular, he proves that a for every function on a subset of vertices of a graph, there exists a \textit{low-frequency} function on the whole graph that coincides with the given function on the subset. A function is called low-frequency if it is a linear combination of Laplace eigenfunctions with small eigenvalues. Other approaches to spectral to sampling are presented in a survey on Graph Signal Processing~\cite[Chapter III.E]{shuman2013emerging}. However, none of the mention techniques seem to be have an immediate connection with the notion of graphical designs. Thus it is probably not too bold to claim that graphical designs provide a new view on the question of graph sampling.

\subsubsection*{Structure of the paper.} The following section, Section~\ref{sec:main}, is devoted to the main results. In Section~\ref{sec:examples}, we present examples of families of graphs and graphical designs on them. Section~\ref{sec:proofs} contains proofs of the results. In Section~\ref{sec:weak-prod}, we prove a result on graphical designs in the weak product of graphs. We conclude with open questions in Section~\ref{sec:open-q}.

\section{Main Results}\label{sec:main}
In this section we prove main results of this paper, connecting the question of graphical designs to independent sets in Subsection~\ref{subsec:ind}, and to isoperimetric partitions in Subsection~\ref{subsec:part}. In the following section, we apply these results to show that various families of graphs have extremal graphical designs. This solves positively an open question posed in ~\cite{steinerberger2018generalized}: there are natural families of graphs supporting high-quality designs. 

\subsection{Independent Sets}\label{subsec:ind}
A subset $S\subset V$ of vertices in $G$ is called independent if no two vertices in it are connected by an edge. The independence ratio $\alpha(G)$ of $G$ is defined as
\[
 \alpha(G) = \max \left\{\frac{|S|}{|V|}\mid S\subset V \text{ is independent}\right\}.
\]
The following spectral upper bound on the size of an independent set was proved by Hoffman in~\cite{hoffman1970eigenvalues} (see also~\cite{hoffman2003eigenvalues}). 
\begin{figure}[h!]\centering
 \begin{tikzpicture}
\renewcommand*{\VertexLineWidth}{1pt}%vertex thickness
\renewcommand*{\EdgeLineWidth}{0.5pt}% edge thickness
\GraphInit[vstyle=Hasse]
 \tikzset{VertexStyle/.style = {
shape = circle,
fill = white,%
inner sep = 0pt,
outer sep = 0pt,
minimum size = 6pt,
draw}}

\Vertex[LabelOut=true,L=\hbox{$\left\{2, 3\right\}$},x=1.5782cm,y=2.3053cm]{v2}
\Vertex[LabelOut=true,L=\hbox{$\left\{2, 6\right\}$},x=0.0cm,y=2.4005cm]{v4}
\Vertex[LabelOut=true,L=\hbox{$\left\{2, 5\right\}$},x=3.9637cm,y=1.7574cm]{v5}
\Vertex[LabelOut=true,L=\hbox{$\left\{2, 4\right\}$},x=2.923cm,y=4.759cm]{v6}
\Vertex[LabelOut=true,L=\hbox{$\left\{5, 6\right\}$},x=1.3323cm,y=3.4999cm]{v8}
\Vertex[LabelOut=true,L=\hbox{$\left\{4, 6\right\}$},x=5.0cm,y=1.0096cm]{v9}
\Vertex[LabelOut=true,L=\hbox{$\left\{4, 5\right\}$},x=0.6805cm,y=4.3115cm]{v10}
\Vertex[LabelOut=true,L=\hbox{$\left\{3, 6\right\}$},x=1.9217cm,y=0.85cm]{v11}
\Vertex[LabelOut=true,L=\hbox{$\left\{3, 5\right\}$},x=3.1863cm,y=0.8197cm]{v12}
\Vertex[LabelOut=true,L=\hbox{$\left\{3, 4\right\}$},x=4.1988cm,y=3.8149cm]{v13}

 \tikzset{VertexStyle/.style = {
shape = circle,
fill = black,%
inner sep = 0pt,
outer sep = 0pt,
minimum size = 6pt,
draw}}

\Vertex[LabelOut=true,L=\hbox{$\left\{1, 6\right\}$},x=4.0577cm,y=4.8374cm]{v0}
\Vertex[LabelOut=true,L=\hbox{$\left\{1, 5\right\}$},x=2.9079cm,y=2.2619cm]{v1}
\Vertex[LabelOut=true,L=\hbox{$\left\{1, 4\right\}$},x=0.3814cm,y=0.0cm]{v3}
\Vertex[LabelOut=true,L=\hbox{$\left\{1, 2\right\}$},x=4.948cm,y=2.5213cm]{v14}
\Vertex[LabelOut=true,L=\hbox{$\left\{1, 3\right\}$},x=1.7599cm,y=5.0cm]{v7}
\Edge(v2)(v0)
\Edge(v2)(v1)
\Edge(v3)(v2)
\Edge(v4)(v1)
\Edge(v4)(v3)
\Edge(v5)(v0)
\Edge(v5)(v3)
\Edge(v6)(v0)
\Edge(v6)(v1)
\Edge(v7)(v4)
\Edge(v7)(v5)
\Edge(v7)(v6)
\Edge(v8)(v2)
\Edge(v8)(v3)
\Edge(v8)(v6)
\Edge(v8)(v7)
\Edge(v9)(v1)
\Edge(v9)(v2)
\Edge(v9)(v5)
\Edge(v9)(v7)
\Edge(v10)(v0)
\Edge(v10)(v2)
\Edge(v10)(v4)
\Edge(v10)(v7)
\Edge(v11)(v1)
\Edge(v11)(v3)
\Edge(v11)(v5)
\Edge(v11)(v6)
\Edge(v11)(v10)
\Edge(v12)(v0)
\Edge(v12)(v3)
\Edge(v12)(v4)
\Edge(v12)(v6)
\Edge(v12)(v9)
\Edge(v13)(v0)
\Edge(v13)(v1)
\Edge(v13)(v4)
\Edge(v13)(v5)
\Edge(v13)(v8)
\Edge(v14)(v8)
\Edge(v14)(v9)
\Edge(v14)(v10)
\Edge(v14)(v11)
\Edge(v14)(v12)
\Edge(v14)(v13)
\end{tikzpicture}
\caption{The Kneser $(6,2)$ Graph: it is a graph on the set of all $2$-subsets of a $6$-set, where two subsets are connected by an edge if they are disjoint. It is $6$-regular and has $15$ vertices. There exists an orthonormal basis of eigenfunctions such that all but two of them have mean value zero on the subset of $5$ vertices in black. In other words, these vertices forms an extremal graphical design.}
\end{figure}
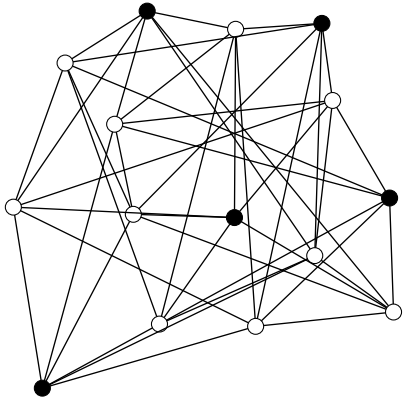
\begin{theorem}[Hoffman, \cite{hoffman1970eigenvalues}]\label{thm:hoffman}
 Let $G$ be a $d$-regular graph on $n$ vertices. The following bound holds
 \begin{equation} 
  \alpha(G)\leq \frac{-\lambda_n}{1 -\lambda_n}, 
 \end{equation}
where $\lambda_n$ is smallest eigenvalue of the normalized adjacency operator $A$ of $G$.
\end{theorem}
The following result connects independent sets with extremal graphical designs.

\begin{theorem}[Main Result]\label{thm:ext-hoffman}
 Let $G$ be a regular graph for which the Hoffman bound is sharp. Let $S\subset V$ be an independent set realizing $\alpha(G)$, i.e.,
 \begin{equation*} 
  \frac{|S|}{|V|}= \alpha(G) = \frac{-\lambda_n}{1-\lambda_n},
 \end{equation*}
then $S$ is an extremal graphical design.
\end{theorem}

In~\cite{steinerberger2018generalized}, the main theorem implies that the neighborhood of a graphical design has to grow exponentially in the number of steps from it. For maximal independent sets, by definition, every vertex of the graph is either in the set or at distance $1$ from it. 

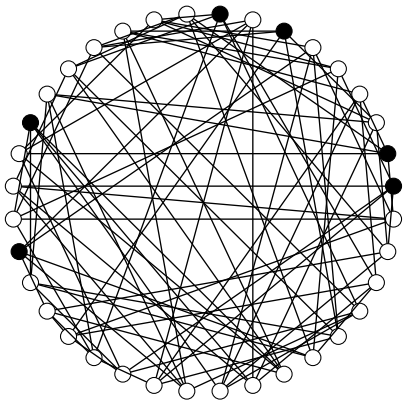
\begin{figure}[h!] \centering
\begin{tikzpicture}
\renewcommand*{\VertexLineWidth}{1pt}%vertex thickness
\renewcommand*{\EdgeLineWidth}{0.5pt}% edge thickness
\GraphInit[vstyle=Hasse]
 \tikzset{VertexStyle/.style = {
shape = circle,
fill = white,%
inner sep = 0pt,
outer sep = 0pt,
minimum size = 6pt,
draw}}

\Vertex[LabelOut=false,L=\hbox{$2$},x=4.7744cm,y=3.5606cm]{v2}
\Vertex[LabelOut=false,L=\hbox{$3$},x=3.9394cm,y=0.4443cm]{v3}
\Vertex[LabelOut=false,L=\hbox{$4$},x=4.5557cm,y=3.9394cm]{v4}
\Vertex[LabelOut=false,L=\hbox{$5$},x=4.2745cm,y=4.2745cm]{v5}
\Vertex[LabelOut=false,L=\hbox{$6$},x=0.7255cm,y=4.2745cm]{v6}
\Vertex[LabelOut=false,L=\hbox{$7$},x=0.2256cm,y=1.4394cm]{v7}
\Vertex[LabelOut=false,L=\hbox{$8$},x=4.924cm,y=1.8505cm]{v8}
\Vertex[LabelOut=false,L=\hbox{$9$},x=3.9394cm,y=4.5557cm]{v9}
\Vertex[LabelOut=false,L=\hbox{$10$},x=0.0cm,y=2.7187cm]{v10}
\Vertex[LabelOut=false,L=\hbox{$11$},x=3.1495cm,y=0.076cm]{v11}
\Vertex[LabelOut=false,L=\hbox{$12$},x=1.8505cm,y=0.076cm]{v12}
\Vertex[LabelOut=false,L=\hbox{$13$},x=2.2813cm,y=0.0cm]{v13}
\Vertex[LabelOut=false,L=\hbox{$14$},x=0.076cm,y=3.1495cm]{v14}
\Vertex[LabelOut=false,L=\hbox{$17$},x=4.7744cm,y=1.4394cm]{v17}
\Vertex[LabelOut=false,L=\hbox{$18$},x=2.2813cm,y=5.0cm]{v18}
\Vertex[LabelOut=false,L=\hbox{$19$},x=1.4394cm,y=0.2256cm]{v19}
\Vertex[LabelOut=false,L=\hbox{$20$},x=1.8505cm,y=4.924cm]{v20}
\Vertex[LabelOut=false,L=\hbox{$21$},x=2.7187cm,y=0.0cm]{v21}
\Vertex[LabelOut=false,L=\hbox{$22$},x=1.0606cm,y=4.5557cm]{v22}
\Vertex[LabelOut=false,L=\hbox{$23$},x=0.7255cm,y=0.7255cm]{v23}
\Vertex[LabelOut=false,L=\hbox{$24$},x=0.4443cm,y=1.0606cm]{v24}
\Vertex[LabelOut=false,L=\hbox{$25$},x=4.5557cm,y=1.0606cm]{v25}
\Vertex[LabelOut=false,L=\hbox{$26$},x=1.0606cm,y=0.4443cm]{v26}
\Vertex[LabelOut=false,L=\hbox{$27$},x=4.2745cm,y=0.7255cm]{v27}
\Vertex[LabelOut=false,L=\hbox{$28$},x=0.0cm,y=2.2813cm]{v28}
\Vertex[LabelOut=false,L=\hbox{$30$},x=1.4394cm,y=4.7744cm]{v30}
\Vertex[LabelOut=false,L=\hbox{$31$},x=3.5606cm,y=0.2256cm]{v31}
\Vertex[LabelOut=false,L=\hbox{$33$},x=0.4443cm,y=3.9394cm]{v33}
\Vertex[LabelOut=false,L=\hbox{$34$},x=5.0cm,y=2.2813cm]{v34}
\Vertex[LabelOut=false,L=\hbox{$35$},x=3.1495cm,y=4.924cm]{v35}
 \tikzset{VertexStyle/.style = {
shape = circle,
fill = black,%
inner sep = 0pt,
outer sep = 0pt,
minimum size = 6pt,
draw}}
\Vertex[LabelOut=false,L=\hbox{$0$},x=5.0cm,y=2.7187cm]{v0}
\Vertex[LabelOut=false,L=\hbox{$1$},x=4.924cm,y=3.1495cm]{v1}
\Vertex[LabelOut=false,L=\hbox{$15$},x=2.7187cm,y=5.0cm]{v15}
\Vertex[LabelOut=false,L=\hbox{$16$},x=3.5606cm,y=4.7744cm]{v16} 
\Vertex[LabelOut=false,L=\hbox{$29$},x=0.076cm,y=1.8505cm]{v29}
\Vertex[LabelOut=false,L=\hbox{$32$},x=0.2256cm,y=3.5606cm]{v32}
\Edge(v0)(v5)
\Edge(v0)(v8)
\Edge(v0)(v10)
\Edge(v0)(v12)
\Edge(v0)(v19)
\Edge(v1)(v4)
\Edge(v1)(v14)
\Edge(v1)(v18)
\Edge(v1)(v33)
\Edge(v1)(v34)
\Edge(v2)(v8)
\Edge(v2)(v11)
\Edge(v2)(v16)
\Edge(v2)(v20)
\Edge(v2)(v33)
\Edge(v3)(v5)
\Edge(v3)(v7)
\Edge(v3)(v17)
\Edge(v3)(v18)
\Edge(v3)(v29)
\Edge(v4)(v12)
\Edge(v4)(v17)
\Edge(v4)(v21)
\Edge(v4)(v30)
\Edge(v5)(v20)
\Edge(v5)(v28)
\Edge(v5)(v30)
\Edge(v6)(v10)
\Edge(v6)(v15)
\Edge(v6)(v18)
\Edge(v6)(v27)
\Edge(v6)(v31)
\Edge(v7)(v19)
\Edge(v7)(v27)
\Edge(v7)(v32)
\Edge(v7)(v33)
\Edge(v8)(v18)
\Edge(v8)(v21)
\Edge(v8)(v24)
\Edge(v9)(v20)
\Edge(v9)(v21)
\Edge(v9)(v27)
\Edge(v9)(v29)
\Edge(v9)(v34)
\Edge(v10)(v11)
\Edge(v10)(v13)
\Edge(v10)(v34)
\Edge(v11)(v17)
\Edge(v11)(v32)
\Edge(v11)(v35)
\Edge(v12)(v26)
\Edge(v12)(v27)
\Edge(v12)(v35)
\Edge(v13)(v25)
\Edge(v13)(v29)
\Edge(v13)(v30)
\Edge(v13)(v33)
\Edge(v14)(v19)
\Edge(v14)(v20)
\Edge(v14)(v31)
\Edge(v14)(v35)
\Edge(v15)(v17)
\Edge(v15)(v20)
\Edge(v15)(v25)
\Edge(v15)(v26)
\Edge(v16)(v22)
\Edge(v16)(v23)
\Edge(v16)(v27)
\Edge(v16)(v30)
\Edge(v17)(v23)
\Edge(v18)(v22)
\Edge(v19)(v23)
\Edge(v19)(v25)
\Edge(v21)(v25)
\Edge(v21)(v32)
\Edge(v22)(v25)
\Edge(v22)(v28)
\Edge(v22)(v35)
\Edge(v23)(v24)
\Edge(v23)(v34)
\Edge(v24)(v26)
\Edge(v24)(v29)
\Edge(v24)(v31)
\Edge(v26)(v28)
\Edge(v26)(v33)
\Edge(v28)(v32)
\Edge(v28)(v34)
\Edge(v29)(v35)
\Edge(v30)(v31)
\Edge(v31)(v32)
\end{tikzpicture}
\caption{The Sylvester Graph: it is $5$-regular and has $36$ vertices, the subset of $6$ vertices in black forms a maximal independent set and a graphical design of order at most 9. In other words, there exists an orthonormal basis of eigenfunctions such that at least $26$ of the non-constant eigenfunctions have mean value zero on the subset of black vertices.} 
\end{figure}

Notably, almost all explicit examples of designs in~\cite{steinerberger2015sharp} have the property of being at distance at most $1$ from every other vertex in the graph, however only one, the one for the Sylvester graph, is a maximal independent set. A computer experiment in Sage on the Sylvester graph shows that there exist maximal independent sets of bigger size (namely, $11$ vertices), however their order as of a graphical design seems to be higher. It cannot be translated directly into precise statement, as the experiment included a choice of basis of eigenfunctions, which is not unique due to the high multiplicity of eigenvalues.

\subsection{Isoperimetric Partitions}\label{subsec:part}
In~\cite{cheeger1969lower}, Cheeger proved a lower bound on the second smallest eigenvalue of the Laplacian of a Riemannian manifold in terms of its isoperimetric constant, which now sometimes is called the Cheeger constant. In~\cite{buser1982note}, Buser proved an upper bound. Later, an analog of the Buser inequality for graphs was proved independently in~\cite{dodziuk1984difference, dodziuk150combinatorial, alon1986eigenvalues}, and an analog of the Cheeger inequality for graphs was proved independently in~\cite{alon1985lambda1, tanner1984explicit}. Usually the Cheeger constant of a graph is defined as
\[
 h'(G) = \min_{\emptyset \neq S\subsetneq V}\frac{|E(S,V\setminus S)|}{\min\left\{|S|,|V\setminus S|\right\}},
\]
where $E(S,V\setminus S)$ stands for the set of edges with one vertex in $S$ and the other one in the complement $ V\setminus S$.
We stick to a slightly different definition
\[
 h(G) = \min_{\emptyset \neq S\subsetneq V}  \frac{|V||E(S,V\setminus S)|}{d|S||V\setminus S|}.
\]
However, these two constants can be approximated one by another. Namely,
\[
 h'(G) \leq d\cdot h(G) \leq 2 h'(G).
\]

A similar modification to the isoperimetric constant is made for Euclidean domains in~\cite{steinerberger2015sharp}, and for simplicial complexes in~\cite{parzanchevski2016isoperimetric}.
We make use of the following theorem, which is a slight variation of the Cheeger inequality for graphs.

\begin{theorem}[Alon-Milman, Tanner \cite{alon1985lambda1, tanner1984explicit}]\label{thm:cheeger}
 Let $G$ be a connected $d$-regular graph, then
 \[
  1-\lambda_2 \leq h(G).
 \]
\end{theorem}

Similarly to the case of the Hoffman theorem, we can deduce a result on extremal graphical designs from the above theorem.
\begin{theorem}\label{thm:ext-cheeger}
Let $G$ be a regular graph for which the Cheeger bound is sharp. Let $S\subset V$ be a set realizing $h(G)$, i.e.,  \begin{equation*} 
  \frac{|V||E(S,V\setminus S)|}{d|S||V\setminus S|} = 1-\lambda_2.
 \end{equation*}
 Then it is an extremal graphical design.
\end{theorem}

\section{Examples}\label{sec:examples}

The Hoffman bound has become a classical tool in spectral graph theory, and has been generalized to other settings.

\subsection{Bipartite Graphs} An immediate example of graphs for which the Hoffman bound is sharp is bipartite graphs, i.e., graphs whose vertex set can be partitioned into two parts in such a way that every edge has exactly one vertex in each part. In this case, $\lambda_n=-1$ (consider a function which equal $1$ on one part, and $-1$ on the other), and the Hoffman bound is sharp for either of the parts of the partition. Therefore, each part of a bipartite graph makes an extremal graphical design.

\subsection{Kneser Graphs}  In extremal combinatorics, the Hoffman bound provided a method of proving results in the spirit of the Erd\H{o}s-Ko-Rado Theorem, including the original theorem. For an overview of this subject, see, e.g.,~\cite{godsil2016erdos}. The original Erd\H{o}s-Ko-Rado Theorem reads as follows.

\begin{theorem}[Erd\H{o}s-Ko-Rado,~\cite{erdos1961intersection}]
 Let $\mathcal{F}$ be a family of distinct subsets of size $k$ of $[n]=\{1,\dots, n\}$ such that each pair of subsets in the family has a non-empty intersection. Then
 \[
  |\mathcal{F}|\leq {n-1 \choose k-1}.
 \]
\end{theorem}
An example of such a family of the maximal size is the family of all subsets of size $k$ containing a fixed element. In~\cite{lovasz1979shannon}, Lov{\'a}sz provided a spectral proof of Erd\H{o}s-Ko-Rado Theorem by considering the corresponding Kneser graph $KG(n,k)$: its vertices are all subsets of size $k$ of $[n]$, and two vertices are connected if the sets do not intersect. Then the desired families of subsets correspond to the independent sets of the Kneser graph. For this graph, the Hoffman bound is sharp. This implies in particular that a family of all subsets containing a fixed element forms an extremal graphical design in the Kneser graph.

\subsection{Derangement Graphs} A similar question for permutations was first addressed by Deza and Frankl in~\cite{frankl1977maximum}. Two permutations $\sigma,\tau\in S_n$ are called intersecting, if there exits $i\in[n] = \{1,\dots,n\}$ such that $\sigma(i) = \tau(i)$. A family of permutations is called intersecting, if any two permutations in it are intersecting.

\begin{theorem}[Deza-Frankl,~\cite{frankl1977maximum}]
 Let $I\subset S_n$ be an intersecting family of permutations, then 
 \[
  |I| \leq (n-1)!.
  \]
\end{theorem}
\begin{figure}[h!]\centering
\begin{tikzpicture}
    \renewcommand*{\VertexLineWidth}{1pt}%vertex thickness
    \renewcommand*{\EdgeLineWidth}{0.5pt}% edge thickness
    \GraphInit[vstyle=Hasse]
    \tikzset{VertexStyle/.style = {
    shape = circle,
    fill = white,%
    inner sep = 0pt,
    outer sep = 0pt,
    minimum size = 6pt,
    draw}}
    \Vertex[LabelOut=false,L=\hbox{$\left\{2, 3\right\}$},x=6.3514cm,y=3.4247cm]{v3}
    \Vertex[LabelOut=false,L=\hbox{$\left\{2, 7\right\}$},x=3.6066cm,y=1.2487cm]{v5}\Vertex[LabelOut=false,L=\hbox{$\left\{2, 6\right\}$},x=0.0cm,y=3.6809cm]{v6}
    \Vertex[LabelOut=false,L=\hbox{$\left\{2, 5\right\}$},x=1.5778cm,y=2.155cm]{v7}
    \Vertex[LabelOut=false,L=\hbox{$\left\{2, 4\right\}$},x=5.7707cm,y=4.3186cm]{v8}
    \Vertex[LabelOut=false,L=\hbox{$\left\{6, 7\right\}$},x=1.1925cm,y=4.7121cm]{v10}
    \Vertex[LabelOut=false,L=\hbox{$\left\{5, 7\right\}$},x=4.94cm,y=0.6508cm]{v11}
    \Vertex[LabelOut=false,L=\hbox{$\left\{5, 6\right\}$},x=5.1674cm,y=2.0436cm]{v12}
    \Vertex[LabelOut=false,L=\hbox{$\left\{4, 7\right\}$},x=7.0cm,y=4.9356cm]{v13}
    \Vertex[LabelOut=false,L=\hbox{$\left\{4, 6\right\}$},x=6.1221cm,y=1.0046cm]{v14}
    \Vertex[LabelOut=false,L=\hbox{$\left\{4, 5\right\}$},x=2.1552cm,y=6.1329cm]{v15}
    \Vertex[LabelOut=false,L=\hbox{$\left\{3, 7\right\}$},x=3.2729cm,y=5.9999cm]{v16}
    \Vertex[LabelOut=false,L=\hbox{$\left\{3, 6\right\}$},x=6.6199cm,y=2.1218cm]{v17}
    \Vertex[LabelOut=false,L=\hbox{$\left\{3, 5\right\}$},x=2.9325cm,y=4.0439cm]{v18}
    \Vertex[LabelOut=false,L=\hbox{$\left\{3, 4\right\}$},x=0.5465cm,y=2.7948cm]{v19}

    \tikzset{VertexStyle/.style = {
    shape = circle,
    fill = black,%
    inner sep = 0pt,
    outer sep = 0pt,
    minimum size = 6pt,
    draw}}
    \Vertex[LabelOut=false,L=\hbox{$\left\{1, 7\right\}$},x=3.4544cm,y=3.0009cm]{v0}
    \Vertex[LabelOut=false,L=\hbox{$\left\{1, 6\right\}$},x=4.7931cm,y=7.0cm]{v1}
    \Vertex[LabelOut=false,L=\hbox{$\left\{1, 5\right\}$},x=4.1349cm,y=4.7281cm]{v2}
    \Vertex[LabelOut=false,L=\hbox{$\left\{1, 4\right\}$},x=1.7328cm,y=0.2376cm]{v4}
    \Vertex[LabelOut=false,L=\hbox{$\left\{1, 3\right\}$},x=3.2144cm,y=0.0cm]{v9}
    \Vertex[LabelOut=false,L=\hbox{$\left\{1, 2\right\}$},x=5.5435cm,y=5.7147cm]{v20}
    \Edge(v3)(v0)
    \Edge(v3)(v1)
    \Edge(v3)(v2)
    \Edge(v4)(v3)
    \Edge(v5)(v1)
    \Edge(v5)(v2)
    \Edge(v5)(v4)
    \Edge(v6)(v0)
    \Edge(v6)(v2)
    \Edge(v6)(v4)
    \Edge(v7)(v0)
    \Edge(v7)(v1)
    \Edge(v7)(v4)
    \Edge(v8)(v0)
    \Edge(v8)(v1)
    \Edge(v8)(v2)
    \Edge(v9)(v5)
    \Edge(v9)(v6)
    \Edge(v9)(v7)
    \Edge(v9)(v8)
    \Edge(v10)(v7)
    \Edge(v10)(v8)
    \Edge(v10)(v9)
    \Edge(v10)(v2)
    \Edge(v10)(v3)
    \Edge(v10)(v4)
    \Edge(v11)(v6)
    \Edge(v11)(v8)
    \Edge(v11)(v9)
    \Edge(v11)(v1)
    \Edge(v11)(v3)
    \Edge(v11)(v4)
    \Edge(v12)(v5)
    \Edge(v12)(v8)
    \Edge(v12)(v9)
    \Edge(v12)(v0)
    \Edge(v12)(v3)
    \Edge(v12)(v4)
    \Edge(v13)(v6)
    \Edge(v13)(v7)
    \Edge(v13)(v9)
    \Edge(v13)(v12)
    \Edge(v13)(v1)
    \Edge(v13)(v2)
    \Edge(v13)(v3)
    \Edge(v14)(v5)
    \Edge(v14)(v7)
    \Edge(v14)(v9)
    \Edge(v14)(v11)
    \Edge(v14)(v0)
    \Edge(v14)(v2)
    \Edge(v14)(v3)
    \Edge(v15)(v5)
    \Edge(v15)(v6)
    \Edge(v15)(v9)
    \Edge(v15)(v10)
    \Edge(v15)(v0)
    \Edge(v15)(v1)
    \Edge(v15)(v3)
    \Edge(v16)(v6)
    \Edge(v16)(v7)
    \Edge(v16)(v8)
    \Edge(v16)(v12)
    \Edge(v16)(v14)
    \Edge(v16)(v15)
    \Edge(v16)(v1)
    \Edge(v16)(v2)
    \Edge(v16)(v4)
    \Edge(v17)(v5)
    \Edge(v17)(v7)
    \Edge(v17)(v8)
    \Edge(v17)(v11)
    \Edge(v17)(v13)
    \Edge(v17)(v15)
    \Edge(v17)(v0)
    \Edge(v17)(v2)
    \Edge(v17)(v4)
    \Edge(v18)(v5)
    \Edge(v18)(v6)
    \Edge(v18)(v8)
    \Edge(v18)(v10)
    \Edge(v18)(v13)
    \Edge(v18)(v14)
    \Edge(v18)(v0)
    \Edge(v18)(v1)
    \Edge(v18)(v4)
    \Edge(v19)(v5)
    \Edge(v19)(v6)
    \Edge(v19)(v7)
    \Edge(v19)(v10)
    \Edge(v19)(v11)
    \Edge(v19)(v12)
    \Edge(v19)(v0)
    \Edge(v19)(v1)
    \Edge(v19)(v2)
    \Edge(v20)(v10)
    \Edge(v20)(v11)
    \Edge(v20)(v12)
    \Edge(v20)(v13)
    \Edge(v20)(v14)
    \Edge(v20)(v15)
    \Edge(v20)(v16)
    \Edge(v20)(v17)
    \Edge(v20)(v18)
    \Edge(v20)(v19)
\end{tikzpicture}
\caption{The Kneser $(7,2)$ Graph: it is $10$-regular and has $21$ vertices. The subset of $6$ vertices in black is an independent set for which the Hoffman bound is sharp, hence it is also an extremal graphical design. That is, there exists an orthonormal basis of eigenfunctions such that all but two of them have mean value zero on the subset.}
\end{figure}
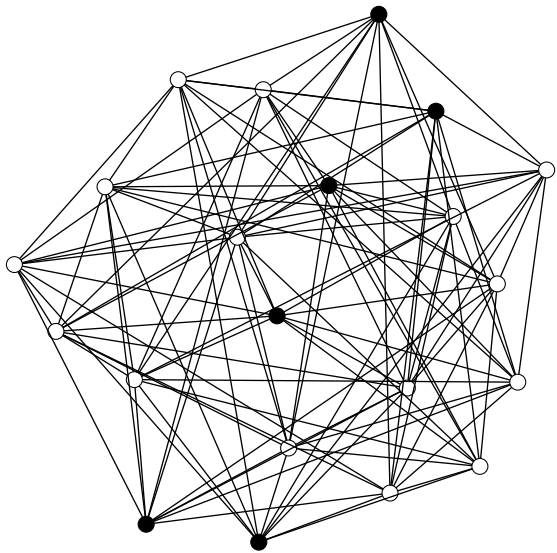

An example of an intersecting family of permutations of the maximal size is the family of permutations fixing a given element. A spectral proof appeared independently in~\cite{renteln2007spectrum},~\cite{friedgut2008intersecting} and~\cite{godsil2009new}. The idea is to consider the derangement graph on the symmetric group $S_n$: its vertices are all the permutations, and two permutations are connected by an edge if they are not intersecting. Then the independent sets of the derangement graph form intersecting families of permutations. The Hoffman bound is sharp for this graph, and therefore, the intersecting families of the maximal size also form extremal graphical designs in this graphs.

\subsection{The Complete Graph} 
Unlike the Hoffman bound, there are not many examples known of graphs for which the Cheeger bound is sharp. The complete graph $K_n$ is the graph on $[n]=\{1,\dots,n\}$, where every two vertices are connected by an edge. The eigenvalues of its normalized adjacency operator are
\[
 \underbrace{\frac{-1}{n-1},\dots,\frac{-1}{n-1}}_{n-1\text{ times}},1,
\]
while its Cheeger constant is equal to 
\[
h(K_n) = \min_{\emptyset \neq A\subsetneq [n]}\frac{|V||E(A,V\setminus A)|}{d |A| |V\setminus A|} = \min_{\emptyset \neq A\subsetneq [n]}\frac{|V||A||V\setminus A|}{d |A| |V\setminus A|} = \frac{n}{n-1},
\]
and the independence ratio is 
\[
 \alpha(K_n) = \frac{1}{n}.
\]
It implies that every proper subset of vertices in $K_n$ is an extremal graphical design.

\subsection{The Hypercube Graph} The hypercube graph $Q_n$ is a graph constructed as follows. Its vertex set is $\{0,1\}^n$, i.e., the set of all $0,1$-vectors of length $n$, where two vectors are connected by an edge iff they differ in exactly one coordinate. The weight of a vector is the sum of its coordinates. Another approach is to view the vertex set as the power set of $[n]=\{1,\dots,n\}$, where two sets are connected by an edge iff their symmetric difference consists of exactly one element. The weight of a vertex is then just the cardinality of the set.
In particular, $Q_n$ is $d$-regular with $d=n$. It is also bipartite, as it can be partitioned into two parts with the vertices with even and odd weights. 
\begin{figure}[h!]
 \centering
 \begin{tikzpicture}
    \renewcommand*{\VertexLineWidth}{1pt}%vertex thickness
    \renewcommand*{\EdgeLineWidth}{0.5pt}% edge thickness
    \GraphInit[vstyle=Hasse]
    \tikzset{VertexStyle/.style = {
    shape = circle,
    fill = white,%
    inner sep = 0pt,
    outer sep = 0pt,
    minimum size = 6pt,
    draw}}
\Vertex[LabelOut=false,L=\hbox{$0000$},x=1.4645cm,y=0.0cm]{v0}
\Vertex[LabelOut=false,L=\hbox{$0001$},x=0.0cm,y=1.4645cm]{v1}
\Vertex[LabelOut=false,L=\hbox{$0010$},x=1.4645cm,y=2.0711cm]{v2}
\Vertex[LabelOut=false,L=\hbox{$0011$},x=0.0cm,y=3.5355cm]{v3}
\Vertex[LabelOut=false,L=\hbox{$0100$},x=2.9289cm,y=1.4645cm]{v4}
\Vertex[LabelOut=false,L=\hbox{$0101$},x=1.4645cm,y=2.9289cm]{v5}
\Vertex[LabelOut=false,L=\hbox{$0110$},x=2.9289cm,y=3.5355cm]{v6}
\Vertex[LabelOut=false,L=\hbox{$0111$},x=1.4645cm,y=5.0cm]{v7}

    \tikzset{VertexStyle/.style = {
    shape = circle,
    fill = black,%
    inner sep = 0pt,
    outer sep = 0pt,
    minimum size = 6pt,
    draw}}
\Vertex[LabelOut=false,L=\hbox{$1000$},x=3.5355cm,y=0.0cm]{v8}
\Vertex[LabelOut=false,L=\hbox{$1001$},x=2.0711cm,y=1.4645cm]{v9}
\Vertex[LabelOut=false,L=\hbox{$1010$},x=3.5355cm,y=2.0711cm]{v10}
\Vertex[LabelOut=false,L=\hbox{$1011$},x=2.0711cm,y=3.5355cm]{v11}
\Vertex[LabelOut=false,L=\hbox{$1100$},x=5.0cm,y=1.4645cm]{v12}
\Vertex[LabelOut=false,L=\hbox{$1101$},x=3.5355cm,y=2.9289cm]{v13}
\Vertex[LabelOut=false,L=\hbox{$1110$},x=5.0cm,y=3.5355cm]{v14}
\Vertex[LabelOut=false,L=\hbox{$1111$},x=3.5355cm,y=5.0cm]{v15}
\Edge(v0)(v1)
\Edge(v0)(v2)
\Edge(v0)(v4)
\Edge(v0)(v8)
\Edge(v1)(v3)
\Edge(v1)(v5)
\Edge(v1)(v9)
\Edge(v2)(v3)
\Edge(v2)(v6)
\Edge(v2)(v10)
\Edge(v3)(v7)
\Edge(v3)(v11)
\Edge(v4)(v5)
\Edge(v4)(v6)
\Edge(v4)(v12)
\Edge(v5)(v7)
\Edge(v5)(v13)
\Edge(v6)(v7)
\Edge(v6)(v14)
\Edge(v7)(v15)
\Edge(v8)(v9)
\Edge(v8)(v10)
\Edge(v8)(v12)
\Edge(v9)(v11)
\Edge(v9)(v13)
\Edge(v10)(v11)
\Edge(v10)(v14)
\Edge(v11)(v15)
\Edge(v12)(v13)
\Edge(v12)(v14)
\Edge(v13)(v15)
\Edge(v14)(v15)
\end{tikzpicture}
\caption{The Hypercube Graph $Q_4$: it is $4$-regular and has $16$ vertices (the power set of $[4]$). The subset of $8$ vertices in black achieves the Cheeger bound and forms an extremal graphical design: there exists an orthonormal basis of eigenfunctions such that all but two of them have mean value zero on the subset. This subset consists of all subsets of $[4]$ who contain a fixed element.}
\end{figure}
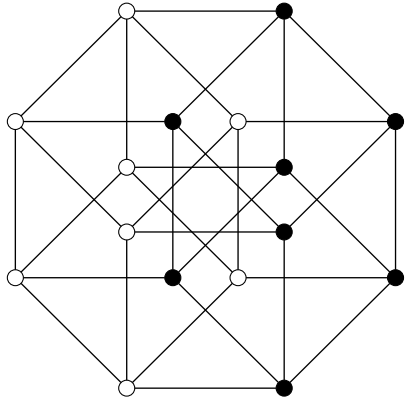

Both the eigenvalues and an orthogonal basis of eigenfunctions of the normalized adjacency operator can be written down explicitly (e.g., see~\cite{harary1988survey}). Namely, for each $I\subseteq [n]$, the following function
\[
 \chi_I (J) = (-1)^{|I\cap J|},
\]
where $J\subseteq [n]$, is an eigenfunction with eigenvalue $1-{2|I|}/{n}$. For $I_1\neq I_2$, the functions $\chi_{I_1}$ and $\chi_{I_2}$ are perpendicular. Hence, they form a basis of eigenfunctions. The Cheeger constant $h(Q_n)$ of the graph is achieved at the following partition
\[
 S = \{I\subseteq [n]\mid 1\in I\},\, T = \{I\subseteq [n]\mid 1\not\in I\}.
\]
Note that every vertex in $S$ has exactly one neighbor in $T$. Hence, for this partition, we have
\[
 \frac{|V||E(S,V\setminus S)|}{d |S| |V\setminus S|} = \frac{2^{n} 2^{n-1}}{n 2^{n-1} 2^{n-1}} = \frac{2}{n}.
\]
The Cheeger bound is sharp for this partition, as the second largest eigenvalue is equal to $1-{2}/{n}$. 

A bit more can be said about extremal graphical designs in the hypercube graphs. 

\begin{theorem}\label{thm:hypercube}
 Let $I\subsetneq [n]$ be a proper subset of $[n]$, then the sets
 \[
  S_I = \{J\subseteq [n]\mid |I\cap J| \text{ is odd}\},   T_I = \{J\subseteq [n]\mid |I\cap J| \text{ is even}\}
 \]
 are extremal graphical designs in $Q_n$.
\end{theorem}

\section{Proofs}\label{sec:proofs}
The proofs of the Hoffman and Cheeger inequalities presented below can probably be considered folklore. We reproduce them here in order to track the case of equality exploited in the proofs of Theorems~\ref{thm:ext-hoffman} and~\ref{thm:ext-hoffman}.

\begin{proof}[Proof of~Theorem~\ref{thm:hoffman}.]
Let $1_S$ be the characteristic function of an independent set $S\subset V$, i.e. such that $\alpha(G) = |S|/|V|$. Since $S$ is independent, the following holds
\[
\langle A 1_S, 1_S \rangle = 0.
\]

As before, let $\lambda_n\leq \dots \leq \lambda_1$ denote the eigenvalues of $A$, and let $\f_n,\dots,\f_1$ be an orthonormal basis of eigenfunctions of $A$, where $\f_i$ has eigenvalue $\lambda_i$, $i=1,\dots,n$. In particular,
\[
 \f_1 = \frac{1_V}{\sqrt{n}}. 
\]

Let $1_S = \sum_{i=1}^n c_i \f_i$ be the decomposition of $1_S$ with respect to the chosen basis. Then
\[
 c_1 = \langle 1_S, \f_1\rangle = \frac{|S|}{\sqrt{n}},\,\text{and } \sum_{i=1}^n c_i^2 = \langle 1_S, 1_S\rangle = |S|.
\]

All these allow us to derive the following chain of inequalities,
\begin{equation}\label{ineq:hoffman}
\begin{aligned}
0 = &  \langle A 1_S, 1_S \rangle = \left \langle \sum_{i=1}^n c_i \lambda_i \f_i, \sum_{i=1}^n c_i \f_i\right \rangle = \\
   & = \sum_{i=1}^n \lambda_i c_i^2 = c_1^2 +  \sum_{i=2}^n \lambda_i c_i^2 \geq c_1^2 + \lambda_n\sum_{i=2}^n c_i^2  = \\
   &  = c_1^2 + \lambda_n \left( \langle 1_S, 1_S \rangle - c_1^2\right) = \\
   & = \left(1-\lambda_n\right) \frac{|S|^2}{n} + \lambda_n |S|,
\end{aligned} 
\end{equation}
which implies the statement of the theorem.
\end{proof}

\begin{proof}[Proof of Theorem~\ref{thm:ext-hoffman}]In order to show that $S$ in the proof of Theorem~\ref{thm:hoffman} is an extremal design we need to show that $1_S$ is a linear combination of the constant function and one other eigenfunction of $A$. Since the bound of Theorem~\ref{thm:hoffman} is sharp for $S$, the chain of inequalities~(\ref{ineq:hoffman}) has also to be sharp for it. This holds iff for some $k\in\mathbb{N}$ (in the notation of the proof above), 
 \[
  1_S = c_1\f_1 + \sum_{i=n-k}^n c_i\f_i,\text{ and } \lambda_{n-k}=\dots=\lambda_n.
 \]
 Therefore, $1_S = c_1\f_1 + C f$, where $f =  \sum_{i=n-k}^n c_i\f_i$ is an eigenfunction of $A$ with eigenvalue $\lambda_n$.
\end{proof}

\begin{proof}[Proof of Theorem~\ref{thm:cheeger}.]
 Let $V=S\sqcup T$ be a partition of the vertex set into non-empty parts realizing the Cheeger constant, i.e.,
 \[
  h(G) = \frac{|V||E(S,T)|}{d|S||T|}.
 \]
 Let $1_S$ and $1_T$ be the characteristic functions of $S$ and $T$, respectively. As before, let $\f_1,\dots,\f_n$ be an orthonormal basis of eigenfunctions numbered w.r.t. the numbering of the eigenvalues $\lambda_1 > \lambda_2 \geq \dots\geq \lambda_n$. 
 Let $1_S = \sum_{i=1}^n c_i \f_i$ be the decomposition of $1_S$ w.r.t. this basis. Then, since $1_S + 1_T = 1_V$ and $1_V = \sqrt{n} \f_1$, the decomposition of $1_T$ is
 \[
  1_T = (\sqrt{n} - c_1)\f_1 + \sum_{i=2}^n (-c_i)\f_i.
 \]
 In the rest, the proof is similar to that of~\ref{thm:hoffman}. Namely, note first that 
 \[
  c_1 = \langle 1_S,\f_1\rangle = \frac{|S|}{\sqrt{n}},\text{ and } \sum_{i=1}^n c_i^2 = \langle 1_S,1_S \rangle = |S|.
 \]
And then note that
\begin{equation}\label{ineq:cheeger}
\begin{aligned}
\frac{1}{d}|E(S,T)| = &  \langle A 1_S, 1_T \rangle = \\
   & = \left \langle \sum_{i=1}^n c_i \lambda_i \f_i, (\sqrt{n} - c_1)\f_1+ \sum_{i=2}^n (-c_i) \f_i\right \rangle = \\
   &  = c_1(\sqrt{n}-c_1) +  \sum_{i=2}^n \lambda_i (-c_i^2) \\
   & \geq c_1(\sqrt{n}-c_1) +  \lambda_2 \sum_{i=2}^n (-c_i^2)\\
   &  = c_1(\sqrt{n}-c_1) -  \lambda_2 \left(\sum_{i=1}^n c_i^2 - c_1^2\right) = \\
   & = \frac{|S|}{\sqrt{n}}\frac{n-|S|}{\sqrt{n}} - \lambda_2\left(|S| - \frac{|S|^2}{n}\right) = \\
   & = \frac{|S|(n-|S|)}{n}\left(1-\lambda_2\right),
\end{aligned} 
\end{equation}
which completes the proof.
\end{proof}

\begin{proof}[Proof of Theorem~\ref{thm:ext-cheeger}.]
In order to show that $S$ in the proof of Theorem~\ref{thm:cheeger} is an extremal design we need to show that $1_S$ is a linear combination of the constant function and one other eigenfunction of $A$. As in the proof of Theorem~\ref{thm:ext-hoffman}, note that in order for the bound on the Cheeger constant to be sharp, the chain of inequalities~\ref{ineq:cheeger} has to be sharp as well. This implies that there exists $k\in\mathbb{N}$, such that
 \[
  1_A = c_1 \f_1 + \sum_{i=2}^k c_i\f_i,\text{ and } \lambda_2 = \dots = \lambda_k.
 \]
 In this case, $1_A$ is a linear combination of the constant function and an eigenfunction with eigenvalue $\lambda_2$.
\end{proof}

\begin{proof}[Proof of Theorem~\ref{thm:hypercube}.]
 First note, that $S_I$ and $T_I$ make a partition of the vertex set of $Q_n$, hence it is enough to show that one of them is a graphical design.
 The eigenfunction $\chi_I$ takes value $-1$ on $S_I$ and value $1$ on $T_I$. Therefore the characteristic function $1_S$ of the set $S_I$ can be decomposed as
 \[
  1_S = \frac{1}{2}\left(1_V - \chi_I\right).
 \]
\end{proof}

\section{Graphical Designs in Weak Products of Graphs}\label{sec:weak-prod}
Let $G_1$ and $G_2$ be regular graphs on the vertex sets $V_1$ and $V_2$, respectively. Their weak (also known as tensor or Kronecker) product $G_1\times G_2$ is defined as follows: it is a graph on the cartesian product $V_1\times V_2$ as the vertex set, and $(v_1,v_2),(u_1,u_2)\in V_1\times V_2$ are connected by an edge iff $v_i\sim u_i$ in $G_i$ for both $i=1,2$.
Let $A_1$ and $A_2$ be the normalized adjacency operators of $G_1$ and $G_2$, respectively. Then the normalized adjacency operator $A$ of $G_1\times G_2$ is the tensor product of $A_1$ and $A_2$, and its matrix is equal to the Kronecker product of the matrices of $A_1$ and $A_2$. (See~\cite{weichsel1962kronecker}.) This permits a desciption of the eigenvalues and eigenfunctions of $A$ in terms of that of $A_1$ and $A_2$.
If we denote the eigenvalues of $A_i$ by $\lambda_{n_i}^i\leq \dots\leq\lambda_{1}^i$ for $i=1,2$, then the eigenvalues of $A$ are the products of the eigenvalues of $A_1$ and $A_2$, i.e.,
\[
 \lambda_{j_1}^1 \lambda_{j_2}^2,\,j_1=1,\dots,n_1,\,j_2=1,\dots,n_2.
\]
Also, if for $i=1,2$ function $f_i$ is an eigenfunction of $A_i$ with eigenvalue $\mu_i$, then the function $f_1\times f_2$ is an eigenfunction of $A$ with eigenvalue $\mu_1\mu_2$. Here $f_1\times f_2:V_1\times V_2\to \mathbb{R}$ is defined as
\[
 (f_1 \times f_2)(v_1,v_2)  = f_1(v_1)f_2(v_2),
\]
for $(v_1,v_2)\in V_1\times V_2$. 
This allows us to deduce the following result.
\begin{theorem}\label{thm:weak-prod}
 If $W_1\subset V_1$ and $W_2\subset V_2$ are graphical designs in $G_1$ and $G_2$ of order $k_1$ and $k_2$, respectively, then 
 \begin{enumerate}
  \item \label{thm:weak-1}$W_1\times W_2$ is a graphical design in $G_1\times G_2$ of order at most $(k_1+1)(k_2+1)-1$.
  \item \label{thm:weak-2}$W_1\times V_2$ and $V_1\times W_2$ are graphical designs of order $k_1$ and $k_2$, respectively.
  \item \label{thm:weak-3}$W_1^{\times r} = W_1\times \dots\times W_1$ is a graphical design of order at most $(k_1+1)^r-1$ in $G_1^{\times r}$, where $r\in\mathbb{N}$.
 \end{enumerate}
\end{theorem} 

\begin{proof}[Proof of Theorem~\ref{thm:weak-prod}.]
In order to show~(\ref{thm:weak-1}), note that $1_{W_1\times W_2} = 1_{W_1}\times1_{W_2}$, and that if for $i=1,2$, $1_{W_i}$ is a linear combination of $1_{V_i}$ and $k_i$ other eigenfunctions, then $1_{W_1\times W_2}$ is linear combination of $1_{V_1\times V_2}$ and other $(k_1+1)(k_2+1)-1$ eigenfunctions.
The proof of~(\ref{thm:weak-2}) follows the same logic, noting that $1_{V_i}$ is an eigenfunction of $A_i$, for $i=1,2$.
Finally,~(\ref{thm:weak-3}) follows from~(\ref{thm:weak-1}) by applying it repeatedly for $G_2 = G_1$. 
\end{proof}

\subsection{Alon-Dinur-Friedgut-Sudakov Stability Theorem for Weak Products}
Theorem~\ref{thm:weak-prod} was largely inspired by~\cite{alon2004graph}. There it was proved that the Hoffman bound remains sharp under the weak product operation. There they also proved a stability result for independent sets that almost achieve the Hoffman bound in weak powers of a graph, this is Theorem~\ref{thm:adfs-stability} below.
\begin{theorem}\cite[Theorem 1.4]{alon2004graph}
 Let $H$ be a connected $d$-regular graph on $n$ vertices and let $\lambda_n \leq \dots \leq \lambda_1 = 1$ be the eigenvalues of its normalized adjacency operator. If
 \[
  \alpha(H) = \frac{-\lambda_n}{1-\lambda_n},
 \]
 then for every integer $r\geq 1$,
  \[
  \alpha(H^{\times r}) = \frac{-\lambda_n}{1-\lambda_n}.
 \]
Moreover, if $H$ is also non-bipartite, and if $I$ is an independent set of the size $\frac{-\lambda_n}{1-\lambda_n}n^r$ in $H^{\times r}$, then there exists a coordinate $i\in\{1,\dots,r\}$ and a
maximum independent set $J$ in $H$, such that
\[
 I = \{ v\in V^r \mid v_i \in J \}.
\]
\end{theorem}
\begin{theorem}\cite[Theorem 1.5]{alon2004graph}\label{thm:adfs-stability}
 Let $H$ be a connected $d$-regular non-bipartite graph on $n$ vertices and let $\lambda_n \leq \dots \lambda_1 = 1$ be the eigenvalues of its normalized adjacency operator. Assume also that the Hoffman bound is sharp for $H$.
 
 Then there exists a constant $M = M(H)$ such that for any $\e > 0$ the following holds. Let $G = H^{\times r}$ and let $I$ be an independent set in $G$ such that 
 \[
  \frac{|I|}{n^r} = \alpha(H) - \epsilon,
 \]
then there exists an independent set $I'$ in $G$ such that 
 \[
  \frac{|I'|}{n^r} = \alpha(H),\text{ and } \frac{|I\Delta I'|}{n^r} < M\epsilon,
 \]
 where $\Delta$ stands for the symmetric difference of sets.
\end{theorem}
\section{Open Questions and Further Directions}\label{sec:open-q}

\subsection{Hoffman bound and Cheeger inequality in other settings}\label{subsec:l2-hoffman}
The Hoffman bound, Theorem~\ref{thm:hoffman}, was generalized to the setting of bounded operators on $L^2$-functions on a measure space in~\cite{bachoc14}. It was then applied to the questions of coloring and independence ratio in different settings, in particular, a sphere in \cite{decorte2015}, Euclidean space in \cite{bachoc2015}, hyperbolic plane in \cite{DeCorte2019}, and hyperbolic surfaces in \cite{golubev2019periodic}. 
The Hoffman bound was also generalized to hypergraphs (equivalently, simplicial complexes) in~\cite{golubev2016chromatic} and later in~\cite{AnnaGundert2015}. In~\cite{filmus2020high}, a new generalization is presented, which is stable under the operation of weak product of hypergraphs.
As it was pointed out, the Cheeger inequality was originally proved in the setting of Riemannian manifolds. There has been an extensive work on generalizing Cheeger inequality to the setting of simplicial complexes, see~\cite{parzanchevski2016isoperimetric, gundert2014higher, lubotzky2017high,lubotzky2014ramanujan,oppenheim2014local,schneeberger2017cheeger}, and~\cite{szymanski2017higher} for a friendly exposition.

\subsection{Open Questions}

\begin{enumerate}
 \item In Theorem~\ref{thm:adfs-stability} of~\cite{alon2004graph}, the stability result is shown for independent sets in weak powers of graphs. Can a result of similar nature be proven for graphical designs?
 \item The hypercube graph $Q_n$ is the $n$-th cartesian power of the complete graph $K_2$ on two vertices (see~\cite{harary1988survey}). What can be said about graphical designs in the cartesian product?  
 \item The results of this paper can probably be extended to weighted graphs, which are of great use in extremal combinatorics (see e.g.~\cite{friedgut2008measure}).  
 \item Spectral theory of hypergraphs and simplicial complexes has been actively developing in recent years (see e.g.~\cite{lubotzky2017high}). What is a reasonable generalization of the graphical designs to this setting?
 \item Spectral techniques have been applied to the questions of coloring and independence ratio for distance graphs in various settings (as briefly described in Subsection~\ref{subsec:l2-hoffman}). Can these techniques help in constructing discrete designs in the corresponding space?
\end{enumerate}

\subsection*{Acknowledgments} The author is grateful to Stefan Steinerberger for the help in preparation of this paper. 

The author is supported by the SNF grant number 200020\_169106.

\bibliographystyle{plain}
\bibliography{mybib}

\begin{thebibliography}{10}

\bibitem{alon1986eigenvalues}
Noga Alon.
\newblock Eigenvalues and expanders.
\newblock {\em Combinatorica}, 6(2):83--96, 1986.

\bibitem{alon2004graph}
Noga Alon, Irit Dinur, Ehud Friedgut, and Benny Sudakov.
\newblock Graph products, fourier analysis and spectral techniques.
\newblock {\em Geometric \& Functional Analysis GAFA}, 14(5):913--940, 2004.

\bibitem{alon1985lambda1}
Noga Alon and Vitali~D Milman.
\newblock $\lambda$1, isoperimetric inequalities for graphs, and
  superconcentrators.
\newblock {\em Journal of Combinatorial Theory, Series B}, 38(1):73--88, 1985.

\bibitem{bachoc14}
Christine Bachoc, Evan DeCorte, F.M. de~Oliveira, and Frank Vallentin.
\newblock Spectral bounds for the independence ratio and chromatic number\ of
  an operator.
\newblock {\em Israel J. Math}, 202:227–254, 2014.

\bibitem{bachoc2015}
Christine Bachoc, Alberto Passuello, and Alain Thiery.
\newblock The density of sets avoiding distance 1 in euclidean space.
\newblock {\em Discrete \& Computational Geometry}, 53(4):783–808, 2015.

\bibitem{bondarenko2013optimal}
Andriy Bondarenko, Danylo Radchenko, and Maryna Viazovska.
\newblock Optimal asymptotic bounds for spherical designs.
\newblock {\em Annals of mathematics}, pages 443--452, 2013.

\bibitem{buser1982note}
Peter Buser.
\newblock A note on the isoperimetric constant.
\newblock In {\em Annales scientifiques de l'{\'E}cole Normale Sup{\'e}rieure},
  volume~15, pages 213--230, 1982.

\bibitem{cheeger1969lower}
Jeff Cheeger.
\newblock A lower bound for the smallest eigenvalue of the laplacian.
\newblock In {\em Proceedings of the Princeton conference in honor of Professor
  S. Bochner}, pages 195--199, 1969.

\bibitem{DeCorte2019}
Evan DeCorte and Konstantin Golubev.
\newblock Lower bounds for the measurable chromatic number of the hyperbolic
  plane.
\newblock {\em Discrete {\&} Computational Geometry}, 62(2):481–496, Sep
  2019.

\bibitem{decorte2015}
Evan DeCorte and Oleg Pikhurko.
\newblock Spherical sets avoiding a prescribed set of angles.
\newblock {\em International Mathematics Research Notices},
  2016(20):6095–6117, 2015.

\bibitem{delsarte1991spherical}
Philippe Delsarte, Jean-Marie Goethals, and Johan~Jacob Seidel.
\newblock Spherical codes and designs.
\newblock In {\em Geometry and Combinatorics}, pages 68--93. Elsevier, 1991.

\bibitem{dodziuk150combinatorial}
J~Dodziuk and WS~Kendall.
\newblock Combinatorial laplacians and isoperimetric inequality, from local
  times to global geometry, control and physics (coventry, 1984/85), 68--74.
\newblock {\em Pitman Res. Notes Math. Ser}, 150.

\bibitem{dodziuk1984difference}
Jozef Dodziuk.
\newblock Difference equations, isoperimetric inequality and transience of
  certain random walks.
\newblock {\em Transactions of the American Mathematical Society},
  284(2):787--794, 1984.

\bibitem{erdos1961intersection}
Paul Erd{\H{o}}s, Chao Ko, and Richard Rado.
\newblock Intersection theorems for systems of finite sets.
\newblock {\em The Quarterly Journal of Mathematics}, 12(1):313--320, 1961.

\bibitem{filmus2020high}
Yuval Filmus, Konstantin Golubev, and Noam Lifshitz.
\newblock High dimensional hoffman bound and applications in extremal
  combinatorics.
\newblock {\em in preparation}.

\bibitem{frankl1977maximum}
Peter Frankl and Mikhail Deza.
\newblock On the maximum number of permutations with given maximal or minimal
  distance.
\newblock {\em Journal of Combinatorial Theory, Series A}, 22(3):352--360,
  1977.

\bibitem{friedgut2008measure}
Ehud Friedgut.
\newblock On the measure of intersecting families, uniqueness and stability.
\newblock {\em Combinatorica}, 28(5):503--528, 2008.

\bibitem{friedgut2008intersecting}
Ehud Friedgut and Haran Pilpel.
\newblock Intersecting families of permutations: An algebraic approach.
\newblock In {\em Talk given at the Oberwolfach Conference in Combinatorics},
  2008.

\bibitem{gariboldi2018optimal}
Bianca Gariboldi and Giacomo Gigante.
\newblock Optimal asymptotic bounds for designs on manifolds.
\newblock {\em arXiv preprint arXiv:1811.12676}, 2018.

\bibitem{godsil2009new}
Chris Godsil and Karen Meagher.
\newblock A new proof of the erd{\H{o}}s--ko--rado theorem for intersecting
  families of permutations.
\newblock {\em European Journal of Combinatorics}, 30(2):404--414, 2009.

\bibitem{godsil2016erdos}
Christopher Godsil and Karen Meagher.
\newblock {\em Erdos-Ko-Rado theorems: algebraic approaches}.
\newblock Number 149. Cambridge University Press, 2016.

\bibitem{golubev2016chromatic}
Konstantin Golubev.
\newblock On the chromatic number of a simplicial complex.
\newblock {\em Combinatorica}, pages 1--12, 2016.

\bibitem{golubev2019periodic}
Konstantin Golubev.
\newblock On periodic sets avoiding given distance on the hyperbolic plane.
\newblock {\em arXiv preprint arXiv:1908.11613}, 2019.

\bibitem{gundert2014higher}
Anna Gundert and May Szedl{\'a}k.
\newblock Higher dimensional discrete cheeger inequalities.
\newblock {\em arXiv preprint arXiv:1401.2290}, 2014.

\bibitem{AnnaGundert2015}
Anna Gundert and May Szedl{\'a}k.
\newblock Higher dimensional discrete cheeger inequalities.
\newblock {\em JoCG}, 6:54--71, 2015.

\bibitem{harary1988survey}
Frank Harary, John~P Hayes, and Horng-Jyh Wu.
\newblock A survey of the theory of hypercube graphs.
\newblock {\em Computers \& Mathematics with Applications}, 15(4):277--289,
  1988.

\bibitem{hoffman1970eigenvalues}
Alan~J Hoffman.
\newblock On eigenvalues and colorings of graphs, b. harris ed., graph theory
  and its applications, 1970.

\bibitem{hoffman2003eigenvalues}
Alan~J Hoffman.
\newblock On eigenvalues and colorings of graphs.
\newblock In {\em Selected Papers Of Alan J Hoffman: With Commentary}, pages
  407--419. World Scientific, 2003.

\bibitem{lovasz1979shannon}
L{\'a}szl{\'o} Lov{\'a}sz.
\newblock On the shannon capacity of a graph.
\newblock {\em IEEE Transactions on Information theory}, 25(1):1--7, 1979.

\bibitem{lubotzky2014ramanujan}
Alexander Lubotzky.
\newblock Ramanujan complexes and high dimensional expanders.
\newblock {\em Japanese Journal of Mathematics}, 9(2):137--169, 2014.

\bibitem{lubotzky2017high}
Alexander Lubotzky.
\newblock High dimensional expanders.
\newblock {\em arXiv preprint arXiv:1712.02526}, 2017.

\bibitem{mclaren1963optimal}
AD~McLaren.
\newblock Optimal numerical integration on a sphere.
\newblock {\em Mathematics of Computation}, 17(84):361--383, 1963.

\bibitem{oppenheim2014local}
Izhar Oppenheim.
\newblock Local spectral expansion approach to high dimensional expanders.
\newblock {\em arXiv preprint arXiv:1407.8517}, 2014.

\bibitem{parzanchevski2016isoperimetric}
Ori Parzanchevski, Ron Rosenthal, and Ran~J Tessler.
\newblock Isoperimetric inequalities in simplicial complexes.
\newblock {\em Combinatorica}, 36(2):195--227, 2016.

\bibitem{pesenson2008sampling}
Isaac Pesenson.
\newblock Sampling in paley-wiener spaces on combinatorial graphs.
\newblock {\em Transactions of the American Mathematical Society},
  360(10):5603--5627, 2008.

\bibitem{renteln2007spectrum}
Paul Renteln.
\newblock On the spectrum of the derangement graph.
\newblock {\em the electronic journal of combinatorics}, 14(1):82, 2007.

\bibitem{schneeberger2017cheeger}
Gr{\'e}goire Schneeberger.
\newblock A cheeger-buser-type inequality on cw complexes.
\newblock {\em arXiv preprint arXiv:1705.01971}, 2017.

\bibitem{shuman2013emerging}
David~I Shuman, Sunil~K Narang, Pascal Frossard, Antonio Ortega, and Pierre
  Vandergheynst.
\newblock The emerging field of signal processing on graphs: Extending
  high-dimensional data analysis to networks and other irregular domains.
\newblock {\em IEEE signal processing magazine}, 30(3):83--98, 2013.

\bibitem{steinerberger2015sharp}
Stefan Steinerberger.
\newblock Sharp l 1-poincar{\'e} inequalities correspond to optimal
  hypersurface cuts.
\newblock {\em Archiv der Mathematik}, 105(2):179--188, 2015.

\bibitem{steinerberger2018generalized}
Stefan Steinerberger.
\newblock Generalized designs on graphs: Sampling, spectra, symmetries.
\newblock {\em arXiv preprint arXiv:1803.02235}, 2018.

\bibitem{szymanski2017higher}
Bazyli Szyma{\'n}ski.
\newblock Higher dimensional expanders.
\newblock {\em M.Sc. Thesis at the University of Warsaw, Poland}, 2017.

\bibitem{tanner1984explicit}
R~Michael Tanner.
\newblock Explicit concentrators from generalized n-gons.
\newblock {\em SIAM Journal on Algebraic Discrete Methods}, 5(3):287--293,
  1984.

\bibitem{weichsel1962kronecker}
Paul~M Weichsel.
\newblock The {K}ronecker product of graphs.
\newblock {\em Proceedings of the American mathematical society}, 13(1):47--52,
  1962.

\end{thebibliography}

\end{document}